\documentclass[10pt]{article}
\usepackage{amsfonts,amssymb,graphicx,amsmath, amsthm, fouridx, stmaryrd, enumitem}
\usepackage[colorlinks=true,linkcolor=blue,citecolor=blue]{hyperref}

\newtheorem{remark}{Remark}
\newtheorem{lemma}{Lemma}
\newtheorem{corollary}{Corollary}
\newtheorem{proposition}{Proposition}
\newtheorem{theorem}{Theorem}[section]

\newcommand{\N}{\mathcal{N}}
\newcommand{\M}{\mathcal{M}}
\newcommand{\LC}{\tilde{\nabla}}
\renewcommand{\j}{\ensuremath{\mathrm{J} }}
\newcommand{\K}{\ensuremath{\mathrm{K} }}
\newcommand{\G}{\tilde{g}}
\newcommand{\J}{\tilde{\rm{J}}}
\renewcommand{\H}[1][]{\ensuremath{{\mathbb{H}^{#1}} }}

\newcommand{\Wd}{\textup{W}}
\newcommand{\<}{\langle}
\renewcommand{\>}{\rangle}

\newcommand{\kn}{\varowedge}
\newcommand{\C}[1][]{\ensuremath{{\mathbb{C}^{#1}} }}
\newcommand{\R}[1][]{\ensuremath{{\mathbb{R}^{#1}} }}
\renewcommand{\S}[1][]{\ensuremath{{\mathbb{S}^{#1}} }}
\renewcommand{\epsilon}{\varepsilon}
\newcommand{\Rm}{\textup{Rm}}
\newcommand{\RR}{\textup{R}}
\newcommand{\Hol}{\textup{Hol}}
\newcommand{\ric}{\textup{Ric}}
\newcommand{\tric}{\widetilde{\ric}}

\date{}
\def\mail#1#2{{\tt #1}@{\tt #2}}

\title{A canonical structure on the tangent bundle of a pseudo- or para-K\"ahler manifold}
\author{Henri Anciaux\footnote{Universidade de S\~ao Paulo; supported by CNPq (PQ 306154/2011-0) and Fapesp (2011/21362-2), \mail{henri.anciaux}{gmail.com}}, 
Pascal Romon\footnote{Universit\'e Paris-Est Marne-la-Vall\'ee, \mail{pascal.romon}{univ-mlv.fr}}}
\begin{document}
\maketitle

\begin{abstract}
It is a classical fact that the cotangent bundle $T^* \M$ of a differentiable manifold $\M$ 
enjoys a canonical symplectic form $\Omega^*$.
If $(\M,\j,g,\omega)$ is a pseudo-K\"ahler or para-K\"ahler $2n$-dimensional manifold,
we prove that the tangent bundle $T\M$ also enjoys a natural pseudo-K\"ahler or para-K\"ahler structure 
$(\J,\G,\Omega)$, where $\Omega$ is the pull-back by $g$ of $\Omega^*$
and $\G$ is a pseudo-Riemannian metric with neutral signature $(2n,2n)$. 
We investigate the curvature properties of the pair $(\J,\G)$ and prove that: $\G$ is scalar-flat, 
is not Einstein unless $g$ is flat, has nonpositive (resp.\ nonnegative) Ricci curvature 
if and only if $g$ has nonpositive (resp.\ nonnegative) Ricci curvature as well, and is 
locally conformally flat if and only if $n=1$ and $g$ has constant curvature, or $n>2$ and $g$ is flat.
We also check that (i) the holomorphic sectional curvature of $(\J,\G)$ is not constant unless $g$ is flat, 
and (ii) in $n=1$ case, that $\G$ is never anti-self-dual, unless conformally flat.

\medskip
\emph{2010 MSC}:  32Q15, 53D05

\emph{Keywords}: tangent bundle, pseudo-K\"ahler geometry, para-K\"ahler geometry, self-duality and anti-self-duality.
\end{abstract}


\section*{Introduction}

It is a classical  fact that given any differentiable manifold $\M$, its cotangent bundle $T^*\M$ enjoys
 a canonical symplectic structure $\Omega^*$. 
 
Moreover, given  a linear connection $\nabla$ on a manifold $\M$, (e.g.\ the Levi-Civita connection  
of a Riemannian metric), the bundle $TT\M$ splits into a direct sum of two subbundles $H\M$ 
and $V \M$, both isomorphic to $T\M$. This allows to define an almost complex structure  $J$ by setting 
$J(X_h ,X_v) :=(- X_v, X_h)$, where, for $X  \in TT\M = H\M \oplus V\M$, we write 
$X \simeq (X_h,X_v) \in T\M \times T\M$.  
Analogously, one may introduce a natural almost para-complex (or bi-Lagrangian) structure, setting 
$J' (X_h ,X_v):=( X_v, X_h)$.

It is  also well known that the tangent bundle of a Riemannian manifold  $(\M,g)$ can be given 
a natural Riemannian structure, called \emph{Sasaki metric}. 
A simple way to understand this construction, which extends \emph{verbatim} to the case of a 
pseudo-Riemannian metric $g$ with signature $(p,m-p)$, is as follows: 
using the splitting $TT\M = H\M \oplus V\M$, we set:
\[
G \big( (X_h ,X_v),(Y_h ,Y_v) \big) := g(X_h,Y_h) + g(X_v,Y_v).
\]
This metric has signature $(2p,2(m-p))$ and is well behaved with respect to $J$  in two ways: 
(\textit{i}) $G$  is compatible with  $J$, i.e.\ $G(.,.)=G(J.,J.)$, and 
(\textit{ii}) the symplectic form $\Omega:=G(J.,.)$  is nothing but  the pull-back of $\Omega^*$ 
by the musical isomorphism between  $T\M \simeq_g T^*\M$. In other words, the triple $(J,G,\Omega)$ defines 
an ``almost pseudo-K\"ahler'' structure\footnote{We might also define 
an ``almost para-K\"ahler'' structure on $T\M$ by introducing the \emph{para-Sasaki metric} 
\[
G'\big( (X_h ,X_v),(Y_h ,Y_v) \big) := g(X_h,Y_h) - g(X_v,Y_v).
\]
This metric has neutral signature $(m,m)$ ($m$ being the dimension of $\M$), is compatible with $J'$ and 
verifies $\Omega:=-G'(J'.,.)$.} on $T\M$.  

Unfortunately, this construction suffers two flaws: $J$ is not integrable unless $\nabla$ is flat and 
the metric $G$ is somewhat ``rigid'': for example, if $G$ has constant scalar curvature, then $g$ is flat 
(see \cite{MT}). We refer to \cite{BV,YI} and the survey~\cite{GuKa} for more detail on the Sasaki metric.

\bigskip

Another construction can be made in the case where $\M$ is complex (resp.\ para-complex): 
in this case both $T\M$ and $T^*\M$ enjoy a canonical complex (resp.\ para-complex) structure which are 
defined as follows: given a family of holomorphic (resp.\ para-holomorphic\footnote{The terminology \emph{split-holomorphic} is sometimes used.}) 
local charts $\varphi: \M \to {\cal U} \subset \R^{2n}$ on $\M$, 
we define holomorphic (resp.\ para-holomorphic) local charts
$\bar{\varphi}: T\M \to {\cal U} \times \R^{2n}$ by 
$\bar{\varphi}(x,V)=(\varphi(x),d \varphi_x (V)), \, \, \forall (x,V)\in T\M$ 
for the tangent bundle, and
$\bar{\varphi}: T^* \M \to {\cal U} \times \R^{2n}$ by 
$\bar{\varphi}(x,\xi)=(\varphi(x),((d \varphi_x)^t)^{-1} (\xi)), \, \, \forall (x,\xi)\in T^* \M$ 
for the cotangent bundle. 
In the first section, we shall see that if $\M$ is merely almost complex (resp. almost para-complex), 
then a more subtle argument allows to define again a canonical almost complex structure  
(resp. almost para-complex structure) on $T\M$. On the other hand,  we shall prove in the second section 
that if $\M$ is pseudo- or para-K\"ahler, the corresponding structure on $T\M$ 
can also be constructed using the splitting $H\M \oplus V\M$ induced by the Levi-Civita connection 
of the K\"ahlerian metric.

Combining the canonical symplectic structure $\Omega^*$ of $T^* \M$ with the canonical complex 
(resp. para-complex) structure $\J^*$ just defined, it is natural to  introduce a $2$-tensor 
$\G^*$ by the formula 
\[
\G^* := \Omega^* (.,\J^*.).
\]
However, it turns out that $\Omega^*$ is not compatible with $\J^*$, 
since it turns out that $\Omega^*(\J^*.,\J^*.)= -\epsilon \Omega^* $ instead of the required formula 
$\Omega^*(\J^*.,\J^*.) = \epsilon \Omega^*$ 
(here and in the following, in order to deal simultaneously with the complex and para-complex cases, 
 we define $\epsilon$ to be such that $(\J^*)^2 = -\epsilon {\rm Id}$, 
i.e. $\epsilon=1$ in the complex case and $\epsilon=-1$ in the para-complex case).
It follows that the tensor ${\G}^*$ is not symmetric and therefore we failed in constructing a canonical 
pseudo-Riemannian structure on $T^*\M$.

 \bigskip

On the other hand, the same idea works well if one considers, instead of the cotangent bundle,  
the tangent bundle of a pseudo- or para-K\"ahler manifold $(\M,\j,g)$,
thus obtaining  a canonical pseudo- or para-K\"ahler structure. 
The purpose of this note is to investigate in detail this construction and
to study its curvature properties. The results are summarized in the following:

\bigskip

\noindent \textbf{Main Theorem} \em
Let $(\M,\j,g,\omega)$ be a pseudo- or para-K\"ahler manifold.
Then $T\M$ enjoys a natural
pseudo- or para-K\"ahler structure $(\J,\G,\Omega)$ with the following properties:
\begin{itemize}
	\item[---] $\J$ is the canonical complex or para-complex structure of $T\M$ induced from that of $\M$;
		\item[---]
	$\Omega$ is the pull-back of $\Omega^*$ by the metric isomorphism   $T\M \simeq_g T^*\M$; 
	\item[---] The pseudo-Riemannian metric $\G$ can be recovered from $\J$ and $\Omega$ by the equation 
	$\G(.,.) := \Omega(.,\J.)$;
	\item[---]
	According to the splitting $TT\M = H \M \oplus V\M$ induced by the Levi-Civita connection of $g$, the triple 
	$(\J,\G,\Omega)$ takes the following expression:
	\begin{eqnarray*}
	 \J (X_h,X_v)&:=&(\j X_h, \j X_v) \\
	 \G \big((X_h,X_v),(Y_h,Y_v)\big)&:=& g(  X_v, \j Y_h)-g( X_h,\j Y_v) \\
	 \Omega \big((X_h,X_v),(Y_h,Y_v)\big)&:=&  g( X_v, Y_h)-g( X_h,Y_v) ;
	 \end{eqnarray*}	
 \item[---] The pseudo-Riemannian metric $\G$ has the following properties:
 \begin{enumerate}[label=(\roman*),align=left]
 \item $\G$ has neutral signature neutral $(2n,2n)$ and is scalar flat;
 \item $(T\M,\G)$ is Einstein if and only if $(\M,g)$ is flat, 
 and therefore $(T\M, \G)$ is flat as well;
 \item the Ricci curvature $\tric$ of $\G$ has the same sign as the Ricci curvature $\ric$ of $g$;
 \item $(T\M,\G)$ is locally conformally flat if and only if $n=1$ and $g$ has constant curvature, 
 or $n > 2$ and $g$ is flat; if $n=1$, $\G$ is always self-dual, so  anti-self-duality is equivalent to 
 conformal flatness;
 \item the pair $(\J,\G)$ has constant holomorphic curvature if and only if $g$ is flat.
\end{enumerate}
\end{itemize} 
\em
\begin{remark}
We use in $(iv)$ the general property that four-dimensional neutral pseudo-K\"ahler or para-K\"ahler manifolds are self-dual if and only if their scalar curvature vanishes. This is analogous to the case of K\"ahler four-dimensional manifolds, except that self-duality is exchanged with anti-self-duality. A proof of this statement is given in Theorem~\ref{thm-W-scal2} in the appendix.
\end{remark}

\bigskip 
 
\noindent This result is a generalization of previous work on the tangent bundle of a Riemannian surface 
(see \cite{GK1}, \cite{GK2}, \cite{AGR}). The authors wish to thank Brendan Guilfoyle for his valuable 
suggestions and comments.

\section{Almost complex and para-complex structures on the tangent bundle} 

Given a manifold $\M$ endowed with an almost complex or almost para-complex
structure $\j$, it is only natural to ask whether its tangent or cotangent
bundle inherit such a structure. The answer is positive:

\begin{proposition}
  Let $(\M,\j)$ be an almost complex (resp. para-complex) manifold. Then its
  tangent bundle  admits a canonical almost complex
  (resp.\ para-complex) structure $\J$. Furthermore, if $\j$ is complex (resp.\ para-complex),  so is~$\J$.
\label{almost_structure}
\end{proposition}

\begin{remark}
  Such a result has been proven already by Lempert \& Sz\"oke \cite{LS} for the
  tangent bundle in the almost complex case. Their construction uses the jets over $\M$ and is quite a bit more
  technical than our proof. However it gives an  interesting interpretation of the meaning of $\J$. 
  We shall see below in Proposition \ref{para} a different and simpler way of defining and
  understanding $\J$, provided $\M$ is a pseudo- or para-K\"ahler
  manifold.
\end{remark}

\begin{proof}
We prove the result using coordinate charts, which amounts to showing that
$\J$ can be defined independently of any change of variable. Let $y = \varphi (x)$ 
be a local change of coordinates on $\mathbb{R}^n$ and write $\xi$ and $\eta$ respectively for the tangent
coordinates induced by the charts (i.e. $\sum_i \xi^i \partial / \partial
x^i = \sum_i \eta^i \partial / \partial y^i$). The change of tangent coordinates at $x$ 
is $\xi \mapsto \eta = d \varphi (x) \xi$, in other words $\varphi$ induces a chart $\Phi$ 
on $\mathbb{R}^{2 n}$, $\Phi : (x, \xi) \mapsto (\varphi (x), d\varphi (x) \xi)$. 
The tangent coordinates at $(x,\xi)$ (resp.\ $(y, \eta)$) are denoted by $(X, \Xi)$ (resp.\ $(Y, H)$) 
and the change of (doubly) tangent coordinates is
\[ 
d \Phi (x, \xi) : (X, \Xi) \mapsto (Y, H) = (d \varphi (x) X, d^2 \varphi(x) (X, \xi) + d \varphi (x) \Xi) .
\]
Assume moreover that we have a $(1, 1)$ tensor, which reads in the $x$ coordinate as the matrix $\j(x)$ 
and in the $y$ coordinate as the matrix $\j' (y) = \j' (\varphi (x)) = d \varphi (x) \circ \j (x) \circ
(d \varphi (x))^{- 1}$. Equivalently for any $X$ and $Y = d \varphi (x) X$, we have 
$\j' (y) Y = \j'(\varphi (x)) d \varphi (x) X = d \varphi (x) \j (x) X$. 
Differentiating this equality along $\xi$ yields
\begin{multline}  \label{Dj}
  (D_{d \varphi (x) \xi} \j') (\varphi(x)) d \varphi (x) X + 
  \j' (\varphi(x)) d^2 \varphi (x) (X, \xi) 
  \\
  = d\varphi (x)  (D_{\xi} \j) (x) X +   d^2 \varphi (x) (\j (x) X, \xi), 
\end{multline}
where $(D_{\xi} \j) (x)$ denotes in this proof the directional derivative of the matrix $\j$ at $x$ 
in the direction $\xi$ (not a covariant derivative).
  
We now define the $(1, 1)$ tensor $\J$ in the $(x,\xi)$ coordinate by
\[ 
\J (x, \xi) : (X, \Xi) \mapsto (\j (x) X, \j (x) \Xi + D_{\xi} \j (x) X) .
\]
Let us prove that this definition is coordinate-independent (for greater
readability we will often write $\j, \j'$ for $\j (x), \j'(y)$). 
Using (\ref{Dj}) and the symmetry of the second order differential $d^2 \varphi (x)$,
\begin{eqnarray*}
  d \Phi (x, \xi)  (\j (X, \Xi)) 
  & = & d \Phi (x, \xi) (\j X, \j \Xi + D_{\xi} \j (x)
  X)
  \\
  & = & (d \varphi (x) \j X, d^2 \varphi (x)
  (\j X, \xi) + d \varphi (x)  (\j \Xi + D_{\xi} \j (x) X))
  \\
  & = & (\j' Y,\j' d \varphi (x) \Xi 
  \\
  && + (D_{d \varphi (x) \xi} \j') (\varphi (x)) d \varphi (x) X + \j' d^2 \varphi (x)(X, \xi))
  \\
  & = & (\j' Y, \j'  (d \varphi (x) \Xi + d^2 \varphi (x)  (X, \xi))
  \\
  && +(D_{d\varphi (x) \xi} \j') (\varphi (x)) d \varphi (x) X)
  \\
  & = & (\j' Y, \j' H + D_{\eta} \j' (y) Y) = \J' (y, \eta)  (Y, H) ,
\end{eqnarray*}
where $\J'$ denotes the map corresponding to $\J$ in the $(y,\eta)$ coordinates. 
Consequently the tensor on $\M$ extends naturally to $T \M$.

\bigskip
  
We have so far defined a $(1,1)$ tensor on $T\M$  without extra assumptions. 
Suppose now that $\j$ is an almost complex (resp. para-complex) structure, so that
$\j^2 = -\varepsilon \text{Id}$. Differentiating this property yields
$\j \,D_{\xi} \j + (D_{\xi} \j) \, \j = 0$. Then
\begin{multline*}
   \J^2  (X, \Xi) = (\j^2 X, \j (\j \Xi + D_{\xi} \j X) + D_{\xi} \j (\j X)) 
   \\
   = (-\varepsilon X, -\varepsilon \Xi + \j (d \j \xi) X + (d \j \xi)  (\j X) 
   = -\varepsilon (X, \Xi)
\end{multline*}
so that $\J$ is also an almost complex (resp.\ para-complex) structure.

Finally if $\j$ is a complex (resp.\ para-complex) structure then we can use complex (resp.\ para-complex) 
coordinate charts, which amounts to saying that $\j$ is a constant matrix. Then $\J$ defined in the associated 
charts on $T \M$ takes a simpler expression, and is also constant:
\[ 
\J (x, \xi) : (X, \Xi) \mapsto (\j X, \j \Xi) 
\]
and that characterizes a complex (resp.\ para-complex) structure.
\end{proof}

\begin{remark}
Finding a similar almost-complex structure on $T^*\M$ is much more difficult, and may not be true 
in all generality. The Reader will note that, whenever $\M$ is endowed with a pseudo-Riemannian metric, 
we have a musical correspondence between $T\M$ and $T^*\M$, and $\J$ induces a corresponding 
structure $\J^*$ on $T^*\M$. However different metrics will yield different structures on $T^*\M$. 
There is one unambiguous case, which will be the setting in the remainder of this article, 
namely when $\j$ is integrable.
\end{remark}

\section{The K\"ahler structure} 

Let $\M$ be a differentiable manifold. We denote by $\pi$ and $\pi^*$ the canonical projections
 $ T \M \to \M$ and $ T \M^* \to \M$. 
 The  subbundle $\ker (d\pi):=V\M$ of $TT \M$
(it is thus a bundle over $T\M$) will be called \em the
vertical bundle. \em

Assume now that $\M$ is equipped with a linear connection $\nabla$.
The corresponding horizontal bundle is defined as follows:
 let $\bar{X}$ be a tangent vector to $T\M$ at some point
$(x_0,V_0)$. This implies that there exists a curve $\gamma(s)=
(x(s),V(s))$ such that $(x(0),V(0))=(x_0,V_0)$ and $\gamma'(0) = \bar{X}$. If
$X \notin V \M$ (which implies $x'(0) \neq 0$), we define the
connection map (see \cite{Do}, \cite{AGR}) $\K : TT \M \to T\M$ by $\K \bar{X} =
\nabla_{x'(0)}V(0)$, where $\nabla$ denotes the Levi-Civita connection of the metric $g$.
If $X$ is vertical, we may assume that the
curve $\gamma$ stays in a fiber so that $V(s)$ is a curve in a
vector space. We then define $\K \bar{X}$ to be simply $V'(0)$.  The
horizontal bundle is then $Ker (\K)$ and we have a direct sum
\begin{equation}
\begin{array}{rcc} TT\M= H\M \oplus V\M & \simeq & T\M \oplus T\M
 \\ \bar{X}& \simeq &  (\Pi \bar{X}, \K \bar{X} ).
  \end{array} \label{eq:directsum}
\end{equation}
Here and in the following, $\Pi$ is a shorthand notation for $d\pi$. 

\begin{lemma} \emph{\cite{Do}} \label{bracket}
Given a vector field ${X}$ on $(\M,\nabla)$ there exists exactly one vector field $X^h$ and one vector field
$X^v$ on $T\M$
such that $(\Pi X^h,\K X^h)=(X,0)$ and $(\Pi X^v,\K X^v)=(0,X)$.
Moreover, given two vector fields $X$ and $Y$ on $(\M,\nabla)$, we have, at the point $(x,V)$:
\begin{eqnarray*} 
\ [X^v,Y^v] & = &  0,      \\
\ [X^h,Y^v] & = & (\nabla_X Y)^v \simeq (0,\nabla_X Y ),\\
\ [X^h,Y^h] &  \simeq & ([X,Y], -\text{\emph{R}}(X,Y)V),
\end{eqnarray*}
where $\text{\emph{R}}$ denotes the curvature of $\nabla$ and we use the direct sum notation $(\ref{eq:directsum})$.
\end{lemma}
The Reader should not confuse the horizontal \emph{lift} $X^h$, which is a vector field on $T\M$ 
constructed from a vector field $X \in \mathfrak{X}(\M)$, with the notation $\bar{X}_h = \Pi \bar{X}$ 
denoting the horizontal \emph{part} of $\bar{X} \in \mathfrak{X}(T\M)$. Similarly, the vertical lift $X^v$ 
is \emph{not} the vertical projection $\bar{X}_v = K \bar{X}$. 

We say that a vector field $\bar{X}$ on $T\M$ is \emph{projectable} if it is constant on the fibres, i.e.\
$(\Pi \bar{X},\K \bar{X})(x,V)=(\Pi \bar{X},\K \bar{X})(x,V')$. According to the
lemma above, it is equivalent to the fact that there exists two vector fields
$X_1$ and $X_2$ on $\M$ such that $\bar{X}=(X_1)^h + (X_2)^v$.

\bigskip

Assume now that $\M$ is equipped with
a pseudo-Riemannian metric $g$, i.e.\ a non-degenerate bilinear form.
By the non-degeneracy assumption, we
 can identify $T^*\M$ with $T\M$ by the following (musical) isomorphism:
 \[ \begin{array}{lccc} \iota :
 & T\M
 &\to& T \M^* \\
& (x,V)& \mapsto &  (x,\xi),
  \end{array}\]
 where $\xi$ is defined by
 \[\xi(W)=g(V,W), \quad \forall \,  W \in T_x\M. \]
The \emph{Liouville form} $\alpha \in  \Omega^1(T^*\M)$ is the $1$-form defined
 by $\alpha_{(x,\xi)}(\bar{X})= \xi_x (d\pi^* (\bar{X}))$, where $\bar{X}$ is a tangent vector at the point 
 $(x,\xi)$ of $T^*\M$. The canonical symplectic form on $T\M^*$ is defined to be $\Omega^*:= - d\alpha$.  
There is an elegant,  explicit formula for the symplectic form $\Omega:=\iota^*(\Omega^*)$ in terms of 
the metric $g$ and the splitting induced by the Levi-Civita connection $\nabla$  (see \cite{An}, \cite{La}):
 \begin{lemma} \label{symplectic}
Let $\bar{X}$ and $\bar{Y}$ be two tangent vectors to
$T \M$; we have
 \[ \Omega(\bar{X},\bar{Y}) = g(\K \bar{X}, \Pi \bar{Y})-g(\Pi \bar{X},\K \bar{Y}).\]
\end{lemma}

\begin{proposition} \label{para} 
Let $(\M,\j,g)$ be a pseudo- or para-K\"ahler manifold. The canonical structure $\J$ satisfies
\[\J \bar{X} \simeq \J (\Pi \bar{X} ,\K \bar{X})= (\j \Pi \bar{X}, \j \K \bar{X}).
\]
\end{proposition}

\begin{corollary} \label{metric} Let $(\M,\j,g)$ be a pseudo- or para-K\"ahler manifold. 
The $2$-tensor $\G(.,.) := \Omega(.,\J.)$ satisfies
\[	 \G ( \bar{X},\bar{Y})= g(  \K \bar{X}, \j \Pi \bar{Y})-g( \Pi \bar{X},\j \K \bar{Y}).\]
 Moreover, $\G$ is symmetric and therefore defines a pseudo-Riemannian metric on $T\M$.
\end{corollary}

\begin{proof}[Proof of Proposition \ref{para}]
  Let us write the splitting of $T T \M$ in a local coordinate $x$ as in the
  proof of Proposition \ref{almost_structure} (\footnote{
  The Reader should be aware of the conflicting notation: the splitting of $T T \M \simeq \R^{4n}$ as
  $\mathbb{R}^{2n} \oplus \mathbb{R}^{2n}$ induced by the coordinate charts 
  (e.g. $\bar{X}\simeq((x,\xi),(X,\Xi))$) differs a priori from the connection-induced splitting 
  $\bar{X} \simeq (\Pi \bar{X}, \K \bar{X})$.}). 
  The Levi--Civita connection is
  expressed through its connection form $\mu$: $\nabla_X Y = d Y (X) + \mu(X) Y$. 
  Consequently, if $(X, \Xi) \in T_{(x, \xi)} T \M$, $\Pi (X, \Xi) = X$ 
  and $\K(X, \Xi) = \Xi + \mu (X) \xi$. Thus  
  \[ \Pi (\J (X, \Xi)) = \j X 
     \text{ and } 
     \K (\J (X, \Xi)) = \j(x) \Xi + (d \j (x) \xi) X + \mu(\j (x) X) \xi . 
  \]
  Because $\j$ is integrable, we may choose $x$ to be a complex coordinate, so
  that $\j$ is a constant endomorphism, and $d \j (x) \xi$
  vanishes. Because $\M$ is K\"ahler, we know that $\mu (X)$
  commutes with $\j$. However, $\nabla$ being without torsion, $\mu(X) Y = \mu (Y) X$, so
  \[
  \K (\J (X, \Xi)) = \j \Xi + \j \mu(X) \xi = \j \K (X, \Xi) . \qedhere
  \]
\end{proof}

\begin{corollary} \label{compatible2}
The symplectic form $\Omega$ is compatible with the complex or para-complex structure $\J$. 
\end{corollary}

\begin{proof}
Using Lemma \ref{symplectic}, we compute
\begin{eqnarray*} \Omega(\J \bar{X},\J \bar{Y})&=&g(\K \J \bar{X}, \Pi \J \bar{Y})-g(\Pi \J \bar{X},\K \J \bar{Y})\\
&=&g(\j \K  \bar{X}, \j \Pi  \Bar{Y})-g(\j  \Pi  \bar{X}, \j \K \bar{Y})\\
&=&\epsilon g( \K  \bar{X},  \Pi  \Bar{Y})-\epsilon g( \Pi  \bar{X},  \K \bar{Y})\\
&=&\epsilon  \Omega(\bar{X},\bar{Y}). 
\end{eqnarray*} 
\end{proof}

\section{The Levi-Civita connection of $\G$}

The following lemma describes the Levi-Civita connection $\LC$ of $\G$ in terms of the direct
decomposition of $TT\M$, the Levi-Civita connection $\nabla$ of $g$ and its curvature tensor $\text{R}$.

\begin{lemma} \label{connection}
Let $\bar{X}$ and $\bar{Y}$ be  two vector fields on $T\M$ and assume that $\bar{Y}$ is projectable, 
then at the point $(x,V)$ we have
\[ (\LC_{\bar{X}} \bar{Y})_{|V} = \left( \nabla_{\Pi \bar{X} } \Pi \bar{Y},\nabla_{\Pi \bar{X} } \K \bar{Y}
        - \text{\emph{T}}_1(\Pi \bar{X},\Pi \bar{Y}, V) \right),
\]
where 
\begin{eqnarray*}
\text{\emph{T}}_1(X,Y,V) & = & \frac{1}{2}\Big( \text{\emph{R}}( X ,  Y)V - \epsilon \text{\emph{R}}(V, \j  X) \j Y - \epsilon \text{\emph{R}}(V, \j   Y) \j X\Big)
\end{eqnarray*}
Moreover, if $\M$  is a pseudo-Riemannian surface with Gaussian curvature $c$, we have
\[
\text{\emph{T}}_1(X,Y,V) = \left\{ \begin{array}{ll}
- 2 c g(V,X) Y  & \text{ in the K\"ahler case, }
\\
+ 2 c g(V,Y) X  & \text{ in the para-K\"ahler case.}
\end{array} \right.
\]

\end{lemma}
\begin{proof}
We use Lemma \ref{bracket} together with the Koszul formula:
\begin{eqnarray*}
2\G(\LC_{\bar{X}} \bar{Y}, \bar{Z}) &=& \bar{X} \G(\bar{Y},\bar{Z}) + \bar{Y} \G(\bar{X}, \bar{Z})
-\bar{Z} \G(\bar{X},\bar{Y}) + \G([\bar{X},\bar{Y}],\bar{Z}) \\
&& -\G([\bar{X},\bar{Z}],\bar{Y})-\G([\bar{Y},\bar{Z}],\bar{X}),
\end{eqnarray*}
where $X$, $Y$ and $Z$ are three vector fields on $T\M$.
From the fact that $[X^v,Y^v]$ and $\G(X^v,Y^v)$ vanish we have:
\begin{eqnarray*}
2\G(\LC_{X^v} Y^v, Z^v) &=& X^v \G(Y^v,Z^v)+Y^v\G(X^v,Z^v)-Z^v\G(X^v,Y^v) \\
&&+\G([X^v,Y^v],Z^v)-\G([X^v,Z^v],Y^v)-\G([Y^v,Z^v],X^v)\\
&=&0 .
\end{eqnarray*}
Moreover, taking into account that $\G(Y^v,Z^h)$ and similar quantities are constant on the fibres,
we obtain
\begin{eqnarray*}
2\G(\LC_{X^v} Y^v, Z^h) &=& X^v \G(Y^v,Z^h)+Y^v\G(X^v,Z^h)-Z^h\G(X^v,Y^v) \\
&& + \G([X^v,Y^v],Z^h)-\G([X^v,Z^h],Y^v)-\G([Y^v,Z^h],X^v) \\
& = & -\G(-(\nabla_Z X)^v,Y^v)-\G(-(\nabla_Z Y)^v,X^v)\\
&=&0.
\end{eqnarray*}
From these last two equations we deduce that $\LC_{X^v} Y^v$ vanishes.
Analogous computations show that $ \LC_{X^v} Y^h$ vanishes as well.
From Lemma \ref{bracket} and the formula $[\bar{X},\bar{Y}]=\LC_{\bar{X}} \bar{Y}-\LC_{\bar{Y}} \bar{X}$, 
we deduce that
 \begin{equation} \label{uno}
 \LC_{X^h} Y^v \simeq (0, \nabla_X Y).
 \end{equation}
Finally, introducing 
\[
\text{T}_1(X,Y,V):=\frac{1}{2}\Big(\text{R}(X,Y)V - \epsilon  \text{R}(V,\j Y)\j X - \epsilon \text{R}(V,\j X)\j Y \Big),
\]
we compute that
\begin{eqnarray*}
2\G(\LC_{X^h} Y^h, Z^h)&=&-g(\text{R}(X,Y)V,\j Z) + g(\text{R}(X,Z)V,\j Y) + g(\text{R}(Y,Z)V, \j X)\\
&=&-g(\text{R}(X,Y)V,\j Z) + g(\text{R}(V, \j Y)X,Z)+ g(\text{R}(V,\j X)Y,Z)\\
&=& -g(\text{R}(X,Y)V,\j Z) + \epsilon g(\text{R}(V,\j Y)\j X,\j Z) + \epsilon g(\text{R}(V,\j X)\j Y,\j Z) \\
&=& - g(2\text{T}_1(X,Y,V),\j Z)
\end{eqnarray*}
and
\[
\G(\LC_{X^h} Y^h, Z^v) = - g(\nabla_X Y , \j Z),
\]
from which we deduce that
\begin{equation} \label{dos}
\LC_{X^h} Y^h (V) = (\nabla_X Y, - \text{T}_1(X,Y,V)).
\end{equation}
From (\ref{uno}) and (\ref{dos}) we deduce the required formula for $\LC_{\bar{X}} \bar{Y}$.
 
\medskip

If $n=1$, we have $\text{R}(X,Y)Z= c(g(Y,Z)X - g(X,Z)Y)$, hence the tensor $\text{T}_1$ becomes:
\begin{eqnarray*}
2 \text{T}_1(X,Y,V) & = & \text{R}(X,Y)V + \epsilon \j \text{R}(V,\j X) Y + \epsilon \j \text{R}(V,\j Y) X 
\\
& = & c \Big( g(Y,V)X - g(X,V)Y 
\\
&& - \epsilon \j \big( g(\j X,Y)V - g(V,Y)\j X  + g(\j Y,X)V - g(V,X)\j Y	 \big) \Big)
\\
& = & c \big( g(Y,V)X - g(X,V)Y 
\\
&& - \epsilon \big( g(\j X,Y) \j V + g(V,Y) X  + g(\j Y,X) \j V + g(V,X) Y \big) \big)
\\
& = & c \big( (1-\epsilon ) g(V,Y) X  - (1+\epsilon ) g(V,X) Y \big) .
\end{eqnarray*}
\end{proof}

\begin{remark} \label{vertical_derivation}
It should be noted that covariant derivatives with respect to a projectable vertical field $X^v$
always vanish.
\end{remark}

\begin{proposition}		
The structure $\J$ is parallel with respect to $\LC$.
\end{proposition}
\begin{proof}
It can be seen as a trivial consequence of the fact that $\J$ is complex (resp.\ para-complex) and 
$\Omega$ is closed, but can also be checked directly, using the equivariance properties of $\j$ w.r.t.\ 
 the connection $\nabla$ and the curvature tensor $\rm{R}$. Using the definition of $\J$ 
and Lemma~\ref{connection}, $\LC_{\bar{X}} \J \bar{Y}$ is obvious provided 
$\text{T}_1(X,\j Y,V) = \j \text{T}_1(X,Y,V)$. That is indeed the case since
\begin{multline*}
2 (\text{T}_1(X,\j Y,V) - \j \text{T}_1(X,Y,V)) 
\\
\begin{split}
= & \text{R} \left( X, \j Y \right) V + \text{R} \left( V, \j X \right) Y + \text{R} \left( V, Y
  \right) \j X
\\
&   - \text{R} \left( X, Y \right) \j V - \text{R} \left( V, \j X \right) Y - \text{R} \left( V, \j Y
  \right) X
\\
= & \text{R} \left( X, \j Y \right) V + \text{R} \left( \j Y, V \right) X 
+ \j \left( \text{R} \left( V, Y \right) X + \text{R} \left( Y, X \right) V \right)\\
= &  \text{R} \left( V, X \right) \j Y + \j \text{R} \left( X, V \right) Y = 0, 
\end{split}
\end{multline*}
where we have used Bianchi's identity.
\end{proof}

\section{Curvature properties of $(\J,\G)$}

\subsection{The Riemannian curvature tensor of $\G$}

\begin{proposition} \label{curvature}
The curvature  tensor $\widetilde{\Rm}:=-\G(\tilde{\text{R}}.,.)$ of $\G$
at $(x,V)$ is given by the formula
\begin{eqnarray*}
\widetilde{\Rm} (\bar{X},\bar{Y}, \bar{Z}, \bar{W}) 
& = &  g(\text{T}_2(\Pi \bar{X},\Pi \bar{Y},\Pi  \bar{Z},V), \j \Pi \bar{W}) 
\\
&& \, \, - \Rm( \Pi \bar{X},\Pi \bar{Y},\Pi \bar{Z},\j  \K \bar{W}) 
- \Rm( \Pi \bar{X},\Pi \bar{Y},\j  \K \bar{Z}, \Pi \bar{W})
\\
&& \, \, \,  - \Rm(\Pi \bar{X},\j  \K \bar{Y}, \Pi \bar{Z},\Pi \bar{W}) 
+  \Rm(\j  \K \bar{X}, \Pi \bar{Y},\Pi \bar{Z},\Pi \bar{W}),
\end{eqnarray*}
where 
\[
\text{\emph{T}}_2(X,Y,Z,V):= (\nabla_X \text{\emph{T}}_1)(Y,Z,V) - (\nabla_Y \text{\emph{T}}_1)(X,Z,V). 
\]
Moreover,  $(T\M,\G)$ is scalar flat and  the Ricci tensor of $\G$ is
\[
\tric(\bar{X},\bar{Y})=  2 \ric(\Pi \bar{X},  \Pi \bar{Y}).
\]
\end{proposition}

\begin{corollary}
$(T\M,\G)$ is Einstein if and only if $(\M,g)$ is flat. Moreover
$(T\M,\G)$ has nonnegative (resp.\ nonpositive) Ricci curvature if and only if $(\M,g)$ has nonnegative 
(resp.\ nonpositive) Ricci curvature as well. 
\end{corollary}

\begin{proof}[Proof of Proposition \ref{curvature}]
We will compute the curvature tensor for projectable vector fields, and need only do so for the following 
six cases, due to the symmetries of $\widetilde{\Rm}$. Remark \ref{vertical_derivation} simplifies computations 
greatly, since most vertical derivatives vanish, except when the derived vector field is not projectable. 
In particular $\tilde{{\rm R}}(X^v,Y^v)$ vanishes as endomorphism, hence:
\begin{eqnarray*}
\widetilde{\Rm} (X^v,Y^v,Z^v,W^v) & = & 0 
\\
\widetilde{\Rm} (X^v,Y^v,Z^v,W^h) & = & 0 
\\
\widetilde{\Rm} (X^v,Y^v,Z^h,W^v) & = & 0 
\end{eqnarray*}
To obtain the  last three combinations, let us first derive $\tilde{\text{R}}(X^h,Y^h) Z^h$. 
This is more delicate since we have to covariantly differentiate non-projectable quantities. Indeed
\begin{eqnarray*} 
 \tilde{\text{R}}(X^h,Y^h)Z^h &=& \LC_{X^h} \LC_{Y^h} Z^h- \LC_{Y^h} \LC_{X^h} Z^h-\LC_{[X^h,Y^h]} Z^h
\\
& = & \LC_{X^h} ( \nabla_Y Z, - \text{T}_1(Y,Z,V)) - \LC_{Y^h} ( \nabla_X Z,-\text{T}_1(X,Z,V))
\\
&& - \LC_{([X,Y],-\text{R}(X,Y)V)} Z^h
\\
&=&  (\nabla_X \nabla_Y Z, - \text{T}_1(X,\nabla_Y Z,V)) - D_{X^h} (0,\text{T}_1(Y,Z,V))
\\
&& - (\nabla_Y \nabla_X Z, - \text{T}_1(Y,\nabla_X Z,V) ) +  D_{Y^h} (0,\text{T}_1(X,Z,V))
\\
&& - (\nabla_{[X,Y]} Z, - \text{T}_1([X,Y] ,Z,V))
\\
&=&(\rm{R}(X,Y)Z,0) 
\\
&& - \left(0, \text{T}_1(X,\nabla_Y Z,V)- \rm{T}_1(Y,\nabla_X Z,V) - \text{T}_1([ X,Y] ,Z,V) \right)
\\
&& - \LC_{X^h} (0,\text{T}_1(Y,Z,V)) + \LC_{Y^h} (0,\text{T}_1(X,Z,V) )   
\end{eqnarray*}
Recalling the lemma\footnote{Note that computations in \cite{Ko} are done for the Sasaki metric, 
hence direct results do not apply.} in \cite{Ko}, there exists a vector field $U$ on $M$ such that 
$U(x)=V$ and $(\nabla_X U)(x)=0$. Then the vertical lift of $\text{T}_1(Y,Z,U)$ is seen to agree to first order
with 
\[
(x,V) \mapsto (0,\text{T}_1(X(x),Z(x),V))
\]
thus allowing us to use the formula in Lemma \ref{connection}:
\begin{eqnarray*}
\LC_{X^h} (0,\text{T}_1(Y,Z,\cdot)) & = & \LC_{X^h} (\text{T}_1(Y,Z,U)^v )
\\
& = & \left( 0, \nabla_X (\text{T}_1(Y,Z,U)) \right) 
\\
& = & \big( 0, (\nabla_X \text{T}_1)(Y,Z,U) + \text{T}_1(\nabla_X Y,Z,U) 
\\
&& \qquad + \text{T}_1(Y, \nabla_X Z,U)  + \text{T}_1(Y,Z,\nabla_X U) \big) 
\end{eqnarray*}
which, evaluated at $(x,V)$, yields
\[
\LC_{X^h} (0,\text{T}_1(Y,Z,\cdot))|_{(x,V)} 
= (0,(\nabla_X \text{T}_1)(Y,Z,V) + \text{T}_1(\nabla_X Y,Z,V) + \text{T}_1(Y, \nabla_X Z,V)).
\]
Summing up,
\begin{eqnarray*}
\tilde{\rm{R}}(X^h,Y^h)Z^h|_{(x,V)} & = & \Big( \rm{R}(X,Y)Z , 
\\
&& - \text{T}_1(X,\nabla_Y Z,V) + \text{T}_1(Y,\nabla_X Z,V) 
\\
&& + \text{T}_1([X,Y] ,Z,V) -(\nabla_X \text{T}_1)(Y,Z,V)
\\
&& - \text{T}_1(\nabla_X Y,Z,V) - \text{T}_1(Y, \nabla_X Z,V)
\\
&& + (\nabla_Y \text{T}_1)(X,Z,V) + \text{T}_1(\nabla_Y X,Z,V) 
\\
&& + \text{T}_1(X, \nabla_Y Z,V)
\Big)
\\
& = & \Big( \text{R}(X,Y)Z , -(\nabla_X \text{T}_1)(Y,Z,V) + (\nabla_Y \text{T}_1)(X,Z,V) \Big)
\\
& = & \Big( \text{R}(X,Y)Z, -\text{T}_2(X,Y,Z,V) \Big).
\end{eqnarray*}
From that we deduce directly
\begin{eqnarray*}
\widetilde{\Rm} (X^h,Y^h,Z^h,W^v) & = &  - \Rm(X,Y,Z,\j W)
\\
 \widetilde{\Rm} (X^h,Y^h,Z^h,W^h)|_{(x,V)} & = & g(\text{T}_2(X,Y,Z,V), \j W).
\end{eqnarray*}
On the other hand, using repeatedly Remark~\ref{vertical_derivation},
\begin{eqnarray*}
\widetilde{\Rm} (X^h,Y^v,Z^h,W^v) 
& = & \G( \LC_{X^h} \LC_{Y^v} W^v - \LC_{Y^v} \LC_{X^h} W^v - \LC_{[X^h,Y^v]} W^v, Z^h )
\\
& = & \G( - \LC_{Y^v} (0,\nabla_X W), Z^h ) = \G(0,Z^h) = 0.
 \end{eqnarray*}
 The claimed formula is easily deduced using the symmetries of the curvature tensor.

\medskip

In order to calculate the Ricci curvature of $\G$, we consider a Hermitian pseudo-orthonormal basis 
$(e_1,\ldots,e_{2n})$ of $T_x\M$, i.e. $g(e_a,e_b)=\epsilon_a \delta_{ab}$, where $\epsilon_a=\pm 1$, 
and $e_{n+a}=\j e_a$. In particular, $\epsilon_{n+a} =  \epsilon \epsilon_a$.This gives a (non-orthonormal) basis
of $T_{(x,V)}T\M$:
\[\bar{e}_a:=(e_a)_h \quad \quad \bar{e}_{2n+a}:=(e_a)^v.\]
A calculation using Corollary \ref{metric}  shows that the expression of $\G$ in this basis is:
\[
[\G_{\mu \nu}]_{1 \leq \mu,\nu \leq 4n}
:= \left( \begin{array}{cccc}  0&0&0&  \Delta 
\\ 
0&0& -\Delta&0 \\ 0& -\Delta&0&0 \\  \Delta&0&0&0 \end{array} \right),
\]
where $\Delta = \epsilon \text{diag}(\epsilon_1, \ldots,\epsilon_n)=\text{diag}(\epsilon_{n+1}, \ldots,\epsilon_{2n})$.
It follows  that
$\tric(X^v,Y^v)$ and $\tric(X^h,Y^v)$ vanish.

Moreover, noting that $\G^{\mu \nu} = \G_{\mu \nu}$,
\begin{eqnarray*}
\tric(X^h,Y^h)& = &
\sum_{\mu,\nu=1}^{4n} \G^{\mu \nu} \widetilde{\Rm}(X^h,\bar{e}_\mu,Y^h,\bar{e}_\nu) 
\\
& = & \sum_{a=1}^{n} \epsilon \epsilon_a \Big(   \widetilde{\Rm}(X^h,(e_a)^h, Y^h,(\j e_a)^v)
-  \widetilde{\Rm}(X^h,(\j e_a)^h, Y^h,(e_a)^v) 
\\
&& -  \widetilde{\Rm}(X^h,(e_a)^v, Y^h,(\j e_a)^h) +  \widetilde{\Rm}(X^h,(\j e_a)^v, Y^h,(e_a)^h) \Big)
\\
& = & \sum_{a=1}^{n} \epsilon \epsilon_a \Big(- \Rm(X,e_a, Y, \j^2 e_a) + \Rm(X,\j e_a, Y, \j e_a)
\\
&& + \Rm(Y,\j e_a, X, \j e_a) - \Rm(Y,e_a, X, \j^2 e_a) 
\Big)
\\
& = & 2 \sum_{a=1}^{n} \Big( \epsilon_a \Rm(X,e_a, Y, e_a) + \epsilon_{a+n} \Rm(X, e_{a+n}, Y, e_{a+n})
\Big)
\\
&= & 2 \sum_{k=1}^{2n}  \epsilon_k \Rm(X,e_k, Y,e_k ) = 2 \ric(X,Y). \end{eqnarray*}
We see easily that $\tric$ vanishes whenever one of the vectors is along the vertical fiber,
thus the expected formula.

Finally the scalar curvature 
\[
\widetilde{\text{Scal}} = \sum_{\mu, \nu=1}^4 \G^{\mu \nu} \tric(\bar{e}_\mu, \bar{e}_\nu)=0,
\]
 since $\G^{\mu \nu} $ vanishes as soon as both $\bar{e}_\mu, \bar{e}_\nu$ are both horizontal.
\end{proof}


\subsection{The Weyl curvature tensor of $\G$} 
\begin{proposition} \label{propoWeyl}
The Weyl tensor  
$\widetilde{\text{\emph{W}}}$
at $(x,V)$ is given by
\[ \begin{array}{rcl}
\widetilde{\text{\emph{W}}}(\bar{X},\bar{Y},\bar{Z},\bar{W}) 
&=& \widetilde{\rm Rm}( \bar{X}, \bar{Y}, \bar{Z},\bar{W})
\\
&& - \frac{1}{2n-1}\Big(
\ric(\Pi \bar{X},\Pi \bar{Z}) \G(\bar{Y},\bar{W}) +\ric(\Pi \bar{Y},\Pi \bar{W}) \G(\bar{Y},\bar{W})  
\\
&& \, \, \quad \quad \, \,   -\ric(\Pi \bar{X},\Pi \bar{W}) \G(\bar{Y},\bar{Z}) -\ric(\Pi \bar{Y},\Pi \bar{Z}) 
\G(\bar{X},\bar{W})   \Big).
\end{array} 
\]
In particular, if $n=1$,
\[
\widetilde{\text{\emph{W}}}(\bar{X},\bar{Y},\bar{Z},\bar{W}) 
= g(\rm{T}_2(\Pi \bar{X},\Pi \bar{Y}, \Pi \bar{Z},V), \j \Pi \bar{W}).
\]
\end{proposition}

\begin{corollary} \label{coroWeyl}
$(T\M,\G)$ is locally conformally  flat if and only if $n=1$ and $g$ has constant curvature, or $n \geq 2$ 
and $g$ is flat.
\end{corollary}

\begin{remark}This result has been proved in the case $n=1$ and $\epsilon=1$ in \cite{GK1}.
\end{remark}

\begin{proof}[Proof of Proposition \ref{propoWeyl}]
Since the scalar curvature vanishes, we have 
\[
\widetilde{\text{W}} = \widetilde{\Rm} - \frac{1}{4n-2} \tric \kn \G,
\]
where $\kn$ denotes the Kulkarni--Nomizu product. 
Recall that  $\tric(\bar{X},\bar{Y})=0$ if one of the two vectors $\bar{X}$ and $\bar{Y}$ is vertical. Consequently 
\begin{eqnarray*}
\tric \kn \G (\bar{X},\bar{Y},\bar{Z},\bar{W})
&=& 2\left( \ric(\Pi \bar{X},\Pi \bar{Z}) \G(\bar{Y},\bar{W}) +\ric(\Pi \bar{Y},\Pi \bar{W}) \G(\bar{Y},\bar{W})  \right.
\\
&& \, \, \left. -\ric(\Pi \bar{X},\Pi \bar{W}) \G(\bar{Y},\bar{Z}) -\ric(\Pi \bar{Y},\Pi \bar{Z}) \G(\bar{X},\bar{W}) \right).
\end{eqnarray*}
The expression of the Weyl tensor follows easily.
 
In the case $n=1$ of a surface with Gaussian curvature $c$, we have  $\ric(X,Y)= c g(X,Y)$
and $\Rm(X,Y,Z,W)=c\big( g(X,Z)g(Y,W)-g(X,W)g(Y,Z)\big) $. Hence using Proposition \ref{curvature}, 
the expression of Weyl tensor simplifies and we get the claimed formula.
\end{proof}

\begin{proof}[Proof of Corollary \ref{coroWeyl}]
We first deal with the case $n=1$. 
Lemma~\ref{connection} implies that $\text{T}_1(X,Y,Z) = -2 c g(Z,X) Y$ when $\epsilon = 1$ 
(resp.\ $2 c g(Z,Y) X$ when $\epsilon = -1$). Therefore, if $\epsilon=1$,
\begin{eqnarray*}
\text{T}_2(X,Y,Z,W) &=& \nabla_X \text{T}_1 (Y,Z,W) - \nabla_Y \text{T}_1 (X,Z,W)
\\
&=& -2  (X.c) g(W,Y) Z + 2(Y.c) g(W,X) Z 
\\
&=& 2 g \Big( (Y.c) X - (X.c) Y,W \Big) Z,
\end{eqnarray*}
which vanishes if and only if $(X.c) Y = (Y.c) X$ for all vectors $X,Y$, i.e.\ the curvature $c$ is constant. 
Analogously, if $\epsilon=-1$,
\begin{eqnarray*}
\text{T}_2(X,Y,Z,W) &=& \nabla_X \text{T}_1 (Y,Z,W) - \nabla_Y \text{T}_1 (X,Z,W)
\\
&=& 2  (X.c)g(W,Z) Y -2 (Y.c)g( W, Z)X \\
&=& 2 \Big( (X.c)Y -(Y.c) X\Big)g(W,Z), 
\end{eqnarray*}
which again vanishes if and only if the curvature $c$ is constant.
\medskip

Assume now that $(T \M,\G)$ is conformally flat with $n \geq 2$. Thus in particular
\begin{multline*}
\widetilde{\text{W}} (X^h, Y^h, Z^h, W^v) 
\\
\begin{split}
= & - \Rm(X,Y,Z, \j  W) 
 \\
& -  \frac{1}{2n-1} \Big(-\ric(X,Z)g(Y,\j W) + \ric(Y,Z)g(X,\j W)   \Big)
\end{split}
\end{multline*}
vanishes, so
\[\Rm(X,Y,Z,\j W)=\frac{1}{2n-1} \Big(\ric(X,Z)g(Y,\j W) -\ric(Y,Z)g(X,\j W)   \Big).\]
(Observe that this equation always holds if $\M$ is a surface.)
Let us apply the symmetry property of the curvature tensor to this equation with $Z=X$ and $\j W=Y$,
 assuming furthermore that $X$ and $Y$ are two non-null vectors:
\begin{eqnarray*}
0 &=& (2n-1) \big( \text{Rm}(X,Y,X,Y)- \text{Rm}(Y,X,Y,X) \big)
\\
&=&\ric (X,X)g(Y,Y)-\ric(Y,X) g(X,Y) 
\\ 
&&-\ric (Y,Y)g(X,X)+\ric(X,Y) g(Y,X)
\\
&=&\ric (X,X)g(Y,Y)-\ric (Y,Y)g(X,X).
\end{eqnarray*}
Hence
\[
\frac{\ric (X,X)}{g(X,X)}=\frac{\ric (Y,Y)}{g(Y,Y)} \cdotp
\]
The set of non null vectors being dense in $T\M$, it follows by continuity that $g$ is Einstein. We deduce that
\begin{eqnarray*}
\Rm(X,Y,X,Y)&=&\frac{1}{2n-1} \Big(\ric(X,X)g(Y,Y) -\ric(Y,X)g(X,Y)   \Big)
\\ 
&=&c \Big( g(X,X)g(Y,Y)-g(X,Y)g(X,Y) \Big),
\end{eqnarray*}
so $g$ has constant curvature. But since $\M$ is K\"ahler and has dimension $2n \geq 4$, it must be flat.
\end{proof}

Finally, we recall the general result linking the Weyl tensor to the scalar curvature in dimension four: 
for a neutral pseudo-K\"ahler or para-K\"ahler metric, self-duality is equivalent to scalar flatness 
(see Theorem~\ref{thm-W-scal2} in annex). We can therefore conclude
\begin{corollary}
In dimension four $(n=1)$, the metric $\G$ is anti-self-dual if and only the curvature $c$ of $g$ is constant.
\end{corollary}	

\begin{proof}
Thanks to proposition~\ref{curvature}, we know that $\G$ is scalar flat, hence self-dual 
($\Wd^-$ vanishes identically). In order for $\G$ to be also anti-self-dual, the Weyl tensor 
has to vanish completely, which amounts, following corollary \ref{coroWeyl}, to
having constant (sectional) curvature $c$ on $\M$.
\end{proof}


\subsection{The holomorphic sectional curvature of $(\J,\G)$}
\begin{proposition} \label{propHolol}
 $(\J, \G)$ has constant holomorphic sectional curvature  if and only if $g$ is flat.
\end{proposition}

\begin{proof}
Define the holomorphic sectional curvature tensor of $\G$ by 
$\widetilde{\Hol}(\bar{X}):=\widetilde{\Rm}(\bar{X},\J \bar{X},\bar{X}, \J \bar{X})$.
Writing any doubly tangent vector $\bar{X}$ as the sum of a horizontal and a vertical factor, we will compute 
$\widetilde{\Hol}(X^h+Y^v)$. We deduce from Proposition~\ref{curvature} that
$\widetilde{\Rm}$ vanishes whenever two or more entries are vertical. 
Hence, using the antisymmetric properties of the Riemann tensor w.r.t.\ the complex or para-complex structure, 
\begin{eqnarray*}
\widetilde{\Hol}(X^h+Y^v) &=& \widetilde{\Rm}(X^h,\j X^h,X^h, \j X^h)
\\
&& +  \widetilde{\Rm}(X^h,\j X^h,X^h,\j Y^v)+ \widetilde{\Rm}(X^h,\j X^h,Y^v,\j X^h)
\\
&& + \widetilde{\Rm}(X^h,\j Y^v,X^h,\j X^h)+ \widetilde{\Rm}(Y^v,\j X^h,X^h,\j X^h)
\\
&=& \widetilde{\Rm}(X^h,\j X^h,X^h,\j X^h) + 4 \widetilde{\Rm}(X^h,\j X^h,X^h,\j Y^v)
\\
&=& g(\text{T}_2 (X,\j X,X,V) - 4 \epsilon  \text{R}(X,Y)X , \j X) .
\end{eqnarray*}
In particular,
\begin{eqnarray*}
\widetilde{\Hol}(X^v)&=&0
\\
\widetilde{\Hol}(X^h+ X^v) &=&  g(\text{T}_2 (X,\j X,X,V), \j X)
\\
\widetilde{\Hol}(X^h+ (\j X)^v)&=&   g(\text{T}_2 (X,\j X,X,V), \j X) + 4 \epsilon \Hol(X).
\end{eqnarray*}
It follows from the first equation that if $\widetilde{\Hol}$ is constant, it must be zero. Hence, 
from the second and third equation we deduce that $\Hol$ must vanish, i.e.\ $g$ is flat.
\end{proof}


\section{Examples} \label{exPseudo}
The simplest examples where we may apply the construction above is where $(\M,\j,g,\omega)$ is the plane $\R^2$
equipped with  the flat metric $g:=dq_1^2 + \epsilon dq_2^2$ and the complex or para-complex structure $\j$ 
defined by $\j (\partial_{q_1}, \partial_{q_2})=(-\epsilon \partial_{q_2}, \partial_{q_1})$. In other words,  
$\R^2$ is identified with the complex plane $\C$ or the para-complex plane $\mathbb D$. We recall that 
$\mathbb D$, called the algebra of double numbers, is the two-dimensional real vector space $\R^2$ endowed 
with the commutative algebra structure whose product rule is given by
\[ 
(u , v) . (u' , v')= (uu'+vv', uv'+u'v).
\]
The number $(0,1)$, whose square is $(1,0)$ and not $(-1,0)$, will be denoted by $\tau$.

\medskip

We claim that in the complex case
$\epsilon=1$,  the structure $(\J,\G,\Omega)$ just constructed on $T\C$ is equivalent to that of the standard 
complex pseudo-Euclidean plane $(\C^2, \bar{\j},\<.,.\>_2, \omega_1)$, where $\bar{\j}$ is the canonical 
complex structure, $(z_1 =x_1+iy_1,z_2=x_2+iy_2)$ are the canonical coordinates and
\begin{eqnarray*}
 \<.,.\>_2&:=& -dx_1^2-dx_2^2+dx_2^2+dy_2^2 \\ 
\omega_1 &:=& -dx_1\wedge dy_1 + dx_2 \wedge dy_2.
\end{eqnarray*}
To see this, it is sufficient to consider the following  complex change of coordinates
$$\left\{  \begin{array}{lcl}z_1&:=&\frac{\sqrt{2}}{2}((p_1+i p_2)+i (q_1+iq_2))\\
 z_2&:=&\frac{\sqrt{2}}{2}(p_1+ i p_2- i(q_1+iq_2)), \end{array} \right.$$
 which preserves the symplectic form, since we have
\[\omega_1:=-dx_1\wedge dy_1 + dx_2\wedge dy_2=dq_1 \wedge dp_1 + dq_2 \wedge dp_2= \Omega,\]
where $\Omega$ is the canonical symplectic form of $T^*\C \simeq_g T \C$. The metric of a pseudo-K\"ahler 
structure being determined by the complex structure and the symplectic form through the formula 
$\G=\Omega(.,\J.)$, we have the required identification.

\medskip

Analogously, in the para-complex case $\epsilon=-1$, 
the structure  $(\J,\G,\Omega)$  constructed on $ T\mathbb D$ is equivalent to that of the standard 
para-complex plane ($\mathbb D^2$, $\bar{\j}$, $\<.,.\>_*$, $\omega_*$), where $\bar{\j}$ is the canonical 
para-complex structure, $(w_1 =u_1+\tau u_1,w_2=u_2+ \tau y_2)$ are the canonical coordinates and
\begin{eqnarray*}
 \<.,.\>_*&:=& du_1^2 - dv_1^2+du_2^2 - dv_2^2 \\ 
\omega_* &:=& du_1\wedge dv_1 + du_2\wedge dv_2.
\end{eqnarray*}
Here we have to be careful with the identification of $T^* \mathbb D$ with $T \mathbb D$: since the metric $g$ 
is $dq_1^2-dq_2^2$, we have $q_1:= dp_1 \simeq_{g} \partial_{p_1}$ and $q_2:=dp_2 \simeq_g - \partial_{q_2}$. 
Hence $\Omega^*=dq_1 \wedge dp_1 + dq_2 \wedge dp_2$ and $\Omega=dq_1 \wedge dp_1 - dq_2 \wedge dp_2$.
Introducing the change of para-complex coordinates
$$\left\{  \begin{array}{lcl}
w_1 &:=&\frac{\sqrt{2}}{2}((p_1+\tau p_2) - \tau  (q_1+ \tau q_2))
\\
w_2&:=&\frac{\sqrt{2}}{2}(\tau (p_1+ \tau p_2) + (q_1+\tau q_2)),
\end{array} \right.$$
we check that
\[ 
\omega_* = du_1\wedge dv_1 + du_2\wedge dv_2= dq_1 \wedge dp_1 - dq_2 \wedge dp_2 = \Omega,
\]
hence we obtain the identification between 
$ (T\mathbb D, \J ,\G, \Omega)$ and $(\mathbb D^2, \bar{\j}, \<.,.\>_*,\omega_*)$.
Of course the metrics considered in these two examples are flat.

\medskip

The next  simplest examples of pseudo-Riemannian surfaces are the two-dimensional space forms, namely the sphere $\S^2$, 
the hyperbolic plane $\H^2 :=\{x_1^2 + x_2^2 - x_3^2 = -1 \}$ and the de Sitter surface 
$d\S^2 :=\{x_1^2 + x_2^2 - x_3^2 = 1 \}$. Their tangent bundles enjoy a interesting geometric interpretation 
(see \cite{GK1}): the tangent bundle $T\S^2$ is canonically identified with the set of oriented lines 
of Euclidean three-space:
\[
L(\R^3) \ni \{ V+tx | \,t \in \R \} \simeq (x,V-\<V,x\>_0 x) \in T \S^2 .
\]
Analogously, the tangent bundle $T\H^2$ is canonically identified with the set of oriented negative (timelike) lines 
of three-space endowed with the metric $\<.,.\>_1:=dx_1^2 + dx_2^2 - dx_3^2$:
\[
\mathbb L_{1,-}^3 \ni \{ V+tx | \, t \in \R \} \simeq (x,V-\<V,x\>_1 x) \in T \H^2 ,
\]
Finally, the tangent bundle $Td\S^2$ is canonically identified with the set of oriented positive (spacelike) lines 
of three-space endowed with the metric $\<.,.\>_1$:
\[
\mathbb L_{1,+}^3 \ni \{ V+tx | \, t \in \R \} \simeq (x,V-\<V,x\>_1 x) \in T d\S^2 .
\]
Observe that the metric constructed on $T\S^2$ (resp.\ $T\H^2$) has non-negative (resp.\ non-positive) Ricci curvature. 


\appendix

\section{The Weyl tensor in the pseudo-K\"ahler or para-K\"ahler cases}

The Riemann curvature tensor $\Rm$ of a pseudo-Riemannian manifold $\N$
may be seen as a symmetric form $\RR$ on bivectors of $\Lambda^2 T\N$ 
(see \cite{Be} for references). Splitting $\RR$ along the eigenspaces 
$\Lambda^+ \oplus \Lambda^-$ of the Hodge operator $\ast$ on $\Lambda^2 T\N$, 
yields the following block decomposition
\[ \RR = \left(\begin{array}{cc}
     \Wd^+ + \frac{\text{Scal}}{12} I & \text{Z}\\
     \text{Z}^* & \Wd^- + \frac{\text{Scal}}{12} I
   \end{array}\right) \]
where $Z^*$ denotes the adjoint w.r.t. the induced metric on $\Lambda^2 T\N$,
so that $\Wd = \Wd^+ \oplus \Wd^-$, the Weyl tensor seen as a 2-form on $\Lambda^2 T\N$, 
is the traceless, Hodge-commuting part of the Riemann curvature operator $\RR$. 
Hence the following formula
\[ \Wd = \Rm - \frac{1}{2} \ric \kn g + \frac{\text{Scal}}{12} g
   \kn g \, . 
\]
If, additionally, $\N$ is a four dimensional K\"ahler manifold, then 
\begin{theorem}
$\Wd^+$ can be written as a multiple of the scalar curvature by 
a parallel non-trivial 2-form on $\Lambda^2 T\N$.
\end{theorem}
\noindent
See Prop.~2 in \cite{De} for a proof and the explicit formula for the tensor involved. 
We do not need it explicitly since we are only interested in the following
\begin{corollary}
$(\N, g, \j)$ is \emph{anti-self-dual} \emph{($\Wd^+ = 0$)} if and only if 
the scalar curvature vanishes.
\end{corollary}
\noindent
The result extends to the two cases considered in this article: (1) neutral
pseudo-K\"ahler manifolds and (2) para-K\"ahler manifolds, with a slight
twist: $\Wd^+$ is replaced by $\Wd^-$. Precisely:
 
\begin{theorem} \label{thm-W-scal2}
  Let $(\N, g, \j )$ be a four dimensional manifold endowed with a pseudo-K\"ahler 
  neutral metric (respectively a para-K\"ahler metric, necessarily neutral). 
  Then the Weyl tensor $\Wd$ commutes with the Hodge operator and 
  $\N$ is {\emph{self-dual}} ($\Wd^- = 0$) if and only if the scalar curvature vanishes.
\end{theorem}
\noindent
The result for neutral pseudo-K\"ahler manifolds is probably known and 
relates to representation theory (see~\cite{Be} for introduction and references), 
but since we could not find an explicit
proof in the literature{\footnote{On the contrary, some authors seem to imply
that scalar flatness is equivalent to anti-self-duality, see~\cite{DW}).
However this contradiction could possibly come from a
different choice of orientation, which would exchange self-dual with
anti-self-dual.}}, we will give a simple one below. To our knowledge, the
proof for the para-K\"ahler case is new (albeit similar).


\subsection{The pseudo-K\"ahler case}

We will write explicitly the Weyl tensor in a given positively oriented
orthonormal frame, denoted by $( e_1, e_{1'}, e_2, e_{2'})$, where $e_{1'} =
\j e_1$, $e_{2'} = \j e_2$, $g ( e_1) = g ( e_{1'}) = - 1$ and $g ( e_2) = g (
e_{2'}) = + 1$. (For brevity, $g(X)$ denotes the norm $g(X,X)$.) 
The pseudo-metric $g$ extends to bivectors, has signature $(
2, 4)$, and will be again denoted by $g$: $g ( e_a \wedge e_b) = g
( e_a) g ( e_b) - g ( e_a, e_b)^2 = g ( e_a) g ( e_b)$, so that $\mathcal{B}=
( e_1 \wedge e_{1'}, e_1 \wedge e_2, e_1 \wedge e_{2'}, e_{1'} \wedge e_2,
e_{1'} \wedge e_{2'}, e_2 \wedge e_{2'})$ is an orthonormal frame of
$\Lambda^2$, with $g ( e_a \wedge e_b) = - 1$, except for $g ( e_1 \wedge
e_{1'}) = g ( e_2 \wedge e_{2'}) = + 1$. (Note that the other convention,
taking $- g$ does not change the induced metric on $\Lambda^2$.)

Since the volume $e_1 \wedge e_{1'} \wedge e_2 \wedge e_{2'}$ is positively
oriented, we construct an orthonormal eigenbasis for the Hodge star on $\Lambda^2 T\N$:
\[ \left\{ \begin{array}{l}
     E_1^{\pm} = \frac{\sqrt{2}}{2}  ( e_1 \wedge e_{1'} \pm e_2 \wedge
     e_{2'})\\
     E_2^{\pm} = \frac{\sqrt{2}}{2}  ( e_1 \wedge e_2 \pm e_{1'} \wedge
     e_{2'})\\
     E_3^{\pm} = \frac{\sqrt{2}}{2}  ( e_1 \wedge e_{2'} \mp e_{1'} \wedge
     e_2)
   \end{array} \right. \]
so that $\Lambda^{\pm}$ is generated by $E_1^{\pm}, E_2^{\pm}, E_3^{\pm}$.

\medskip

The K\"ahler condition implies 
\[\Rm ( \j X, \j Y, Z, T) = \Rm ( X, Y, Z, T) = \Rm ( X, Y, \j Z, \j T), 
\]
because $\j $ is isometric and parallel. The matrix of the symmetric 2-form $\RR$ 
in the orthonormal frame $\mathcal{B}$ is
\[ \begin{array}{|c|c|c|c|c|c|c|}
     \hline
     & e_{11'} & e_{12} & e_{12'} & e_{1'2} & e_{1'2'} & e_{22'} \\
     \hline
     e_{11'} & \RR_{11' 11'} & \RR_{11' 12} & \RR_{11' 12'} 
     & \begin{array}{c} \RR_{11' 1' 2} = \\ - \RR_{11' 12'} \end{array}
 & \begin{array}{c} \RR_{11' 1' 2'} = \\ \RR_{11' 12} \end{array}
 & \RR_{11' 22'}\\
     \hline
     e_{12} &  & \RR_{1212} & \RR_{1212'} & \begin{array}{c} \RR_{121' 2}= \\ - \RR_{1212'} \end{array} &
     \begin{array}{c} \RR_{131' 2'} = \\ \RR_{1212} \end{array} & \RR_{1222'}\\
     \hline
     e_{12'} &  &  & \RR_{12' 12'} 
     & \begin{array}{c} \RR_{12' 1' 2} = \\ - \RR_{12' 12'} \end{array} &
     \begin{array}{c} \RR_{12' 1' 2'} = \\ \RR_{1212'} \end{array} & \RR_{12' 22'}\\
     \hline
     e_{1'2} &  &  &  & \begin{array}{c} \RR_{1' 21' 2} = \\ \RR_{12' 12'} \end{array} 
     & \begin{array}{c} \RR_{1' 21' 2'} = \\ - \RR_{1212'} \end{array} 
     & \begin{array}{c} \RR_{1' 222'} = \\ - \RR_{12' 22'} \end{array} \\
     \hline
     e_{1'2'} &  &  &  &  
     & \begin{array}{c} \RR_{1' 2' 1' 2'} = \\ \RR_{1212} \end{array} 
     & \begin{array}{c} \RR_{1' 2'22'} = \\ \RR_{1222'} \end{array}\\
     \hline
     e_{22'} &  &  &  &  &  & \RR_{22' 22'}\\
     \hline
   \end{array} 
\]
where $e_{ab}$ stands for $e_a \wedge e_b$, for greater legibility.
We have written the matrix as a table for clarity and to make symmetries more
obvious, and because $\RR$ is symmetric we need only write half the matrix. We
have used the internal symmetries of $\RR$, to choose among equivalent
coefficients the ones lowest in the lexicographic order of the indices.

\medskip

The Weyl tensor satisfies {\emph{some}} of the $\j $-symmetries of $\RR$: indeed
\begin{eqnarray*}
  \ric ( \j X, \j Y) & = & \sum_{i=1}^4 g ( e_i) \Rm ( \j X, e_i, \j Y, e_i) =
  \sum_{i=1}^4 g ( e_i) \Rm ( X, \j e_i, Y, \j e_i)\\
  & = & \sum_{i=1}^4 g ( \j e_i) \Rm ( X, \j e_i, Y, \j e_i) = \ric ( X, Y)
\end{eqnarray*}
because $( \j e_i)$ is again an orthonormal basis. In particular, this
invariance implies $r_{11'} = \ric ( e_1, e_{1'}) = r_{1' 1} = -
r_{11'}$, so $r_{11'}$ vanish (and so does $r_{22'}$). For the
Kulkarni--Nomizu product,
\begin{equation*} \begin{split}
  \ric \kn g ( \j X, Y, Z, T)  = 
  & 
  \ric (\j X, Z) g (Y, T) + \ric (Y, T) g (\j X, Z) 
  \\
  & - \ric (\j X, T) g (Y, Z) - \ric (Y,Z) g (\j X, T)
  \\
   = & - \ric (X, \j Z) g (\j Y, \j T) - \ric (\j Y, \j T) g (X, \j Z) 
   \\
   & + \ric (X, \j T) g (\j Y, \j Z) + \ric (\j Y, \j Z) g (X, \j T)
   \\
   = & - \ric \kn g ( X, \j Y, \j Z, \j T)
\end{split} 
\end{equation*}
so
\[ \ric \kn g ( \j X, \j Y, Z, T) = - \ric \kn g ( X, \j ^2 Y,
   \j Z, \j T) = \ric \kn g ( X, Y, \j Z, \j T) \, .
\]
Hence the following symmetries (fewer than for $\Rm$) in the
coefficients of $\ric \kn g$, $g \kn g$ and $\Rm$, and
therefore $\Wd$:
\[ \begin{array}{|c|c|c|c|c|c|c|}
     \hline
     & e_{1 1'} & e_{1} \wedge e_{2} & e_{1 2'} & e_{1'} \wedge e_{2} & e_{1' 2'} & e_{2 2'}\\
     \hline
     e_{1 1'} & \Wd_{11' 11'} & \Wd_{11' 12} & \Wd_{11' 12'} 
     & \begin{array}{c} \Wd_{11' 1' 2} = \\ - \Wd_{11' 12'} \end{array} 
     & \begin{array}{c} \Wd_{11' 1' 2'} \\ = \Wd_{11' 12} \end{array}
     & \Wd_{11' 22'}\\
     \hline
     e_{1 2} &  & \Wd_{1212} & \Wd_{1212'} & \Wd_{121' 2} & \Wd_{121' 2'} &
     \Wd_{1222'}\\
     \hline
     e_{1 2'} &  &  & \Wd_{12' 12'} & \Wd_{12' 1' 2} 
     & \begin{array}{c} \Wd_{12' 1' 2'} = \\  - \Wd_{121' 2} \end{array} & \Wd_{12' 22'}\\
     \hline
     e_{1' 2} &  &  &  & \begin{array}{c} \Wd_{1' 21' 2} \\ = \Wd_{12' 12'} \end{array} 
     & \begin{array}{c} \Wd_{1' 21' 2'} = \\ - \Wd_{1212'}  \end{array}
     & \begin{array}{c} \Wd_{1' 222'} = \\ - \Wd_{12' 22'} \end{array} \\
     \hline
     e_{1' 2'} &  &  &  &  
     & \begin{array}{c} \Wd_{1' 2' 1' 2'} \\ = \Wd_{1212} \end{array} 
     & \begin{array}{c} \Wd_{1' 2' 22'} \\ = \Wd_{1222'}  \end{array} \\
     \hline
     e_{2 2'} &  &  &  &  &  & \Wd_{22' 22'}\\
     \hline
   \end{array} \]
Expanding on the above eigenbasis of $\Lambda^+ \oplus \Lambda^-$ (which differs
from the one in the positive definite case) yields the following Weyl tensor
coefficients, which we have simplified using the symmetries above (up to a
factor $1 / 2$ due to normalization):
\[ \begin{array}{|c|c|c|c|}
     \hline
     & E_1^+ & E_2^+ & E_3^+\\
     \hline
     E_1^+ & \Wd_{11' 11'} + \Wd_{22' 22'} + 2 \Wd_{11' 22'} & 2 ( \Wd_{11' 12} +
     \Wd_{1222'}) & 2 ( \Wd_{11' 12'} + \Wd_{12' 22'})\\
     \hline
     E_2^+ &  & 2 ( \Wd_{1212} + \Wd_{121' 2'}) & 2 ( \Wd_{1212'} - \Wd_{121' 2})\\
     \hline
     E_3^+ &  &  & 2 ( \Wd_{12' 12'} - \Wd_{12' 1' 2})\\
     \hline
     E_1^- &  &  & \\
     \hline
     E_2^- &  &  & \\
     \hline
     E_3^- &  &  & \\
     \hline
   \end{array} \]
\[ \begin{array}{|c|c|c|c|}
     \hline
     & E_1^- & E_2^- & E_3^-\\
     \hline
     E_1^+ & \Wd_{11' 11'} - \Wd_{22' 22'} & 0 & 0\\
     \hline
     E_2^+ & 2 ( \Wd_{11' 12} - \Wd_{1222'}) & 0 & 0\\
     \hline
     E_3^+ & 2 ( \Wd_{11' 12'} - \Wd_{12' 22'}) & 0 & 0\\
     \hline
     E_1^- & \Wd_{11' 11'} + \Wd_{22' 22'} - 2 \Wd_{11' 22'} & 0 & 0\\
     \hline
     E_2^- &  & 2 ( \Wd_{1212} - \Wd_{121' 2'}) & 2 ( \Wd_{1212'} + \Wd_{121' 2})\\
     \hline
     E_3^- &  &  & 2 ( \Wd_{12' 12'} + \Wd_{12' 1' 2})\\
     \hline
   \end{array} \]
(Again only half the coefficients are written down.)
Further simplifications come from computing $\Wd$, and using
\begin{eqnarray*}
  \text{Scal} & = & - r_{11} - r_{1' 1'} + r_{22} + r_{2' 2'} = 2 ( r_{22} -
  r_{11})\\
  & = & 2 ( - ( - \RR_{11' 11'} + \RR_{1212} + \RR_{12' 12'}) + ( - \RR_{1212} -
  \RR_{1' 21' 2} + \RR_{22' 22'}))\\
  & = & 2 ( \RR_{11' 11'} - 2 ( \RR_{1212} + \RR_{12' 12'}) + \RR_{22' 22'}) \, .
\end{eqnarray*}
First prove that the Hodge star commutes with $\Wd$ by considering $\Wd (
\Lambda^+, \Lambda^-)$:
\begin{eqnarray*}
  \Wd_{11' 11'} & = & \RR_{11' 11'} + \frac{1}{2}  ( r_{11} + r_{1' 1'}) +
  \frac{\text{Scal}}{6} = \RR_{11' 11'} + r_{11} + \frac{\text{Scal}}{6} 
  \\
  & = & \RR_{1212} + \RR_{12' 12'} + \frac{\text{Scal}}{6}
\end{eqnarray*}
\begin{eqnarray*}
\Wd_{22' 22'} & = & \RR_{22' 22'} - \frac{1}{2}  ( r_{22} + r_{2' 2'}) +
   \frac{\text{Scal}}{6} = \RR_{22' 22'} - r_{22} + \frac{\text{Scal}}{6} 
\\
   & = &  \RR_{1212} + \RR_{12' 12'} + \frac{\text{Scal}}{6}
\end{eqnarray*}
so that $\Wd_{11' 11'} - \Wd_{22' 22'} = 0$. Similarly
\[ \Wd_{11' 12} = \RR_{11' 12} + \frac{r_{1' 2}}{2}, \hspace{1em} \Wd_{1222'} =
   \RR_{1222'} + \frac{r_{12'}}{2} = \RR_{1222'} - \frac{r_{1' 2}}{2} \]
so
\[ \Wd_{11' 12} - \Wd_{1222'} = \RR_{11' 12} - \RR_{1222'} + r_{1' 2} = 0 \]
\[ \Wd_{11' 12'} = \RR_{11' 12'} + \frac{r_{1' 2'}}{2} = \RR_{11' 12'} +
   \frac{r_{12}}{2}, \hspace{1em} \Wd_{12' 22'} = \RR_{12' 22'} -
   \frac{r_{12}}{2}, \]
\[ \Wd_{11' 12'} - \Wd_{12' 22'} = \RR_{11' 12'} - \RR_{12' 22'} + r_{12} = 0 . \]
That proves that $\Wd$ is block-diagonal. 

\medskip

The $\Wd^-$ term satisfies
\begin{eqnarray*}
  \Wd_{11' 11'} + \Wd_{22' 22'} - 2 \Wd_{11' 22'} & = & \RR_{11' 11'} + r_{11} +
  \RR_{22' 22'} - r_{22} + \frac{\text{Scal}}{3}
  \\
  && - 2 \RR_{11' 22'}
  \\
  & = & \RR_{11' 11'} + \RR_{22' 22'} - 2 \RR_{11' 22'} - \frac{\text{Scal}}{6}
  \\
  & = & \RR_{11' 11'} + \RR_{22' 22'} - 2 ( \RR_{1212} + \RR_{12' 12'})
  \\
  && - \frac{\text{Scal}}{6} 
  \\
  & = & \frac{\text{Scal}}{2} - \frac{\text{Scal}}{6} =
  \frac{\text{Scal}}{3}
\end{eqnarray*}
using the first Bianchi identity (and the invariance of $\Rm$):
\[ \RR_{11' 22'} = - \RR_{1' 212'} - \RR_{211' 2'} = \RR_{12' 12'} + \RR_{1212} . \]
\begin{eqnarray*}
  \Wd_{1212} - \Wd_{121' 2'} & = & \RR_{1212} + \frac{r_{22} - r_{11}}{2} -
  \frac{\text{Scal}}{6} - \RR_{121' 2'} = \frac{\text{Scal}}{4} -
  \frac{\text{Scal}}{6} 
  \\
  & = & \frac{\text{Scal}}{12}
\end{eqnarray*}
\begin{eqnarray*}
  \Wd_{12' 12'} + \Wd_{12' 1' 2} & = & \RR_{12' 12'} + \frac{\text{Scal}}{4} -
  \frac{\text{Scal}}{6} + \RR_{12' 1' 2} = \frac{\text{Scal}}{12}
\end{eqnarray*}
\[ \Wd_{1212'} + \Wd_{121' 2} = \RR_{1212'} + \frac{r_{22'}}{2} + \RR_{121' 2} -
   \frac{r_{11'}}{2} = \frac{1}{2}  ( r_{22'} - r_{11'}) = 0 . \]
Finally,
\[ \Wd^- = \text{Scal} \left(\begin{array}{ccc}
     1 / 3 &  & \\
    & 1 / 6 & \\
     &  & 1 / 6
   \end{array}\right) = \frac{\text{Scal}}{6} \text{Id} +
   \frac{\text{Scal}}{6} E^-_1 \otimes E^-_1 \]
(and indeed this matrix is traceless w.r.t.\ the pseudo-metric $g$).
One should note that the above expression differs from the Riemannian case, where
\[ \Wd^+ = \text{Scal} \left(\begin{array}{ccc}
     1 / 3 &  & \\
     & - 1 / 6 & \\
     &  & - 1 / 6
   \end{array}\right) = - \frac{\text{Scal}}{6} \text{Id} +
   \frac{\text{Scal}}{3} E_1^+ \otimes E_1^+ . \]
We let the Reader check that in the neutral case, the $\Wd^+$ part is not a
multiple of the scalar curvature, which completes the proof of Theorem~\ref{thm-W-scal2}.


\subsection{The para-K\"ahler case}

The computations are almost identical, but the results differ from the
pseudo-K\"ahler setup, because the para-complex structure $\j $ is now an
anti-isometry: $\RR ( \j X, \j Y)Z = - \RR ( X, Y) Z$. We pick an
orthonormal basis $( e_1, e_{1'}, e_2, e_{2'})$ with 
$e_{1'} = \j e_1$, $e_{2'} = \j e_2$, and $g(e_1) = g ( e_2) = + 1$, 
$g ( e_{1'}) = g ( e_{2'}) = - 1$. The frame
$\mathcal{B}= ( e_1 \wedge e_{1'}, e_1 \wedge e_2, e_1 \wedge e_{2'}, e_{1'}
\wedge e_2, e_{1'} \wedge e_{2'}, e_2 \wedge e_{2'})$ of $\Lambda^2 T\N$ is
also orthonormal w.r.t. the induced metric on $\Lambda^2$, again denoted by
$g$, which has signature $( 2, 4)$: $g ( e_a \wedge e_b) = g ( e_a) g ( e_b) =
- 1$, except for $g ( e_1 \wedge e_2) = g ( e_{1'} \wedge e_{2'}) = + 1$.

An orthonormal eigenbasis for the Hodge operator is the following:
\[ 
\left\{ \begin{array}{l}
     E_1^{\pm} = \frac{\sqrt{2}}{2}  ( e_1 \wedge e_{1'} \mp e_2 \wedge
     e_{2'})\\
     E_2^{\pm} = \frac{\sqrt{2}}{2}  ( e_1 \wedge e_2 \mp e_{1'} \wedge
     e_{2'})\\
     E_3^{\pm} = \frac{\sqrt{2}}{2}  ( e_1 \wedge e_{2'} \mp e_{1'} \wedge
     e_2)
   \end{array} \right. 
\]
where the $E_a^+$ (resp. $E_a^-$) span $\Lambda^+$ (resp. $\Lambda^-$).
(Note the sign differences w.r.t.\ the pseudo-K\"ahler case.)

\medskip

Since $\j $ is anti-isometric and parallel,
\[
\Rm ( \j X, \j Y, Z, T) = - \Rm ( X,Y, Z, T) = \Rm ( X, Y, \j Z, \j T) \, .
\]
Hence the following symmetries of the riemannian curvature operator $\RR$,
expressed in the frame $\mathcal{B}$ (for symmetry reasons and greater
legibility, lower left coefficients are not written in this and the subsequent
matrices):
\[ \begin{array}{|c|c|c|c|c|c|c|}
     \hline
     & e_{1 1'} & e_{12} & e_{1 2'} 
     & e_{1'} \wedge e_{2} & e_{1' 2'} & e_{2 2'}\\
     \hline
     e_{1 1'} & \RR_{11' 11'} & \RR_{11' 12} & \RR_{11' 12'} 
     & \begin{array}{c} \RR_{11' 1' 2} \\ = - \RR_{11' 12'}  \end{array}
     & \begin{array}{c} \RR_{11' 1' 2'} \\ = - \RR_{11' 12}  \end{array} & \RR_{11' 22'}\\
     \hline
     e_{1 2} &  & \RR_{1212} & \RR_{1212'} 
     & \begin{array}{c} \RR_{121' 2} \\ = - \RR_{1212'} \end{array}  
     & \begin{array}{c} \RR_{121' 2'} \\ = - \RR_{1212} \end{array}  & \RR_{1222'}\\
     \hline
     e_{1 2'} &  &  & \RR_{12' 12'} 
     & \begin{array}{c} \RR_{12' 1' 2} \\ = - \RR_{12' 12'} \end{array}  &
     \begin{array}{c} \RR_{12' 1' 2'} \\ = - \RR_{1212'} \end{array}  & \RR_{12' 22'}\\
     \hline
     e_{1' 2} &  &  &  
     & \begin{array}{c} \RR_{1' 21' 2} \\ = \RR_{12' 12'}  \end{array} 
     & \begin{array}{c} \RR_{1' 21' 2'} \\ = \RR_{1212'}  \end{array} 
     & \begin{array}{c} \RR_{1' 222'} \\ = - \RR_{12' 22'} \end{array} 
     \\
     \hline
     e_{1' 2'} &  &  &  &  
     & \begin{array}{c} \RR_{1' 2' 1' 2'} \\ = \RR_{1212}  \end{array} 
     & \begin{array}{c} \RR_{1' 2'22'} \\ = - \RR_{1222'} \end{array} \\
     \hline
     e_{2 2'} &  &  &  &  &  & \RR_{22' 22'}\\
     \hline
   \end{array} \]
(Note again the similarity with the pseudo-K\"ahler case: only a few signs
change.)

\medskip

The Weyl tensor satisfies {\emph{some}} of the $\j $-symmetries of $\Rm$ since
\begin{eqnarray*}
  \ric ( \j X, \j Y) & = & \sum_{i=1}^4 g(e_i) \Rm ( \j X,
  e_i, \j Y, e_i) = \sum_{i=1}^4 g ( e_i) \Rm ( X, \j e_i, Y, \j e_i)\\
  & = & - \sum_{i=1}^4 g ( \j e_i) \Rm ( X, \j e_i, Y, \j e_i) = - \ric ( X,
  Y)
\end{eqnarray*}
since $( \j e_i)$ is also an orthonormal basis. In particular this invariance
implies $r_{1' 1} = r_{11'} = - r_{1' 1}$, so $r_{11'}$ vanishes (and so
does $r_{22'}$). Finally,
\[ 
\frac{\text{Scal}}{2} = r_{11} + r_{22} = - \RR_{11' 11'} + 2 ( \RR_{1212} -
   \RR_{12' 12'}) - \RR_{22' 22'} . 
\]

The Kulkarni--Nomizu product $\ric \kn g$ satisfies
\begin{eqnarray*}
  \ric \kn g ( \j X, Y, Z, T) 
  & = & 
  \ric (\j X, Z) g (Y, T) + \ric (Y, T) g (\j X, Z) 
  \\
  && - \ric (\j X, T) g (Y, Z) - \ric (Y,Z) g (\j X, T)
  \\
  & = & \ric (X, \j Z) g (\j Y, \j T) + \ric (\j Y, \j T) g (X, \j Z) 
  \\
  && - \ric (X, \j T) g (\j Y, \j Z) - \ric (\j Y, \j Z) g (X, \j T)
  \\
  & = & \ric \kn g ( X, \j Y, \j Z, \j T)
\end{eqnarray*}
so
\[ \ric \kn g ( \j X, \j Y, Z, T) = \ric \kn g ( X, \j ^2 Y, \j Z,
   \j T) = \ric \kn g ( X, Y, \j Z, \j T) \]
and the same property holds for $g \kn g$. Hence the following symmetries
(fewer than for $\Rm$) in the coefficients of $\ric \kn g$, $g
\kn g$ and $\Rm$, and therefore $\Wd$:
\[ \begin{array}{|c|c|c|c|c|c|c|}
     \hline
     & e_{1 1'} & e_{1 2} & e_{1 2'} 
     & e_{1' 2} & e_{1' 2'} & e_{2 2'}\\
     \hline
     e_{1 1'} & \Wd_{11' 11'} & \Wd_{11' 12} & \Wd_{11' 12'} 
     & \begin{array}{c} \Wd_{11' 1' 2} = \\ - \Wd_{11' 12'} \end{array} 
     & \begin{array}{c} \Wd_{11' 1' 2'} \\ = - \Wd_{11' 12} \end{array}  & \Wd_{11' 22'}\\
     \hline
     e_{1 2} &  & \Wd_{1212} & \Wd_{1212'} & \Wd_{121' 2} & \Wd_{121' 2'} &
     \Wd_{1222'}\\
     \hline
     e_{1 2'} &  &  & \Wd_{12' 12'} & \Wd_{12' 1' 2} 
     & \begin{array}{c} \Wd_{12' 1' 2'} \\ = \Wd_{121' 2}  \end{array}  & \Wd_{12' 22'}\\
     \hline
     e_{1' 2} &  &  &  & \begin{array}{c} \Wd_{1' 21' 2} \\ = \Wd_{12' 12'} \end{array}  
     & \begin{array}{c} \Wd_{1' 21' 2'} \\ =  \Wd_{1212'}  \end{array} 
     & \begin{array}{c} \Wd_{1' 222'} \\ = - \Wd_{12' 22'} \end{array} \\
     \hline
     e_{1' 2'} &  &  &  &  
     & \begin{array}{c} \Wd_{1' 2' 1' 2'} \\ = \Wd_{1212}  \end{array} 
     & \begin{array}{c} \Wd_{1' 2'22'} \\ = - \Wd_{1222'} \end{array} \\
     \hline
     e_{2 2'} &  &  &  &  &  & \Wd_{22' 22'}\\
     \hline
   \end{array} \]
Let us now express $\Wd$ in the Hodge basis defined earlier, using the above
symmetries (up to a factor $1 / 2$ due to normalization).
\[ \begin{array}{|c|c|c|c|}
     \hline
     & E_1^+ & E_2^+ & E_3^+\\
     \hline
     E_1^+ & \Wd_{11' 11'} + \Wd_{22' 22'} - 2 \Wd_{11' 22'} & 2 ( \Wd_{11' 12} -
     \Wd_{1222'}) & 2 ( \Wd_{11' 12'} - \Wd_{12' 22'})\\
     \hline
     E_2^+ &  & 2 ( \Wd_{1212} - \Wd_{121' 2'}) & 2 ( \Wd_{1212'} - \Wd_{121' 2})\\
     \hline
     E_3^+ &  &  & 2 ( \Wd_{12' 12'} - \Wd_{12' 1' 2})\\
     \hline
     E_1^- &  &  & \\
     \hline
     E_2^- &  &  & \\
     \hline
     E_3^- &  &  & \\
     \hline
   \end{array} \]
\[ \begin{array}{|c|c|c|c|}
     \hline
     & E_1^- & E_2^- & E_3^-\\
     \hline
     E_1^+ & \Wd_{11' 11'} - \Wd_{22' 22'} & 0 & 0\\
     \hline
     E_2^+ & 2 ( \Wd_{11' 12} + \Wd_{1222'}) & 0 & 0\\
     \hline
     E_3^+ & 2 ( \Wd_{11' 12'} + \Wd_{12' 22'}) & 0 & 0\\
     \hline
     E_1^- & \begin{array}{c} \Wd_{11' 11'} + \Wd_{22' 22'} \\ + 2 \Wd_{11' 22'} \end{array} 
     & 0 & 0\\
     \hline
     E_2^- &  & 2 ( \Wd_{1212} + \Wd_{121' 2'}) & 2 ( \Wd_{1212'} + \Wd_{121' 2})\\
     \hline
     E_3^- &  &  & 2 ( \Wd_{12' 12'} + \Wd_{12' 1' 2})\\
     \hline
   \end{array} \]
Only three terms in the off-block-diagonal part are not obviously zero.
\begin{eqnarray*}
  \Wd_{11' 11'} & = & \RR_{11' 11'} - \frac{1}{2}  ( - r_{11} + r_{1' 1'}) -
  \frac{\text{Scal}}{6} = \RR_{11' 11'} + r_{11} - \frac{\text{Scal}}{6} 
\end{eqnarray*}
\[ \Wd_{22' 22'} = \RR_{22' 22'} - \frac{1}{2}  ( - r_{22} + r_{2' 2'}) -
   \frac{\text{Scal}}{6} = \RR_{22' 22'} + r_{22} - \frac{\text{Scal}}{6} \]
but $r_{11} = - \RR_{11' 11'} + \RR_{1212} - \RR_{12' 12'}$ and $r_{22} = \RR_{2121} -
\RR_{21' 21'} - \RR_{22' 22'} = \RR_{1212} - \RR_{12' 12'} - \RR_{22' 22'}$ so that
\[ \Wd_{11' 11'} - \Wd_{22' 22'} = \RR_{11' 11'} - \RR_{22' 22'} + r_{11} - r_{22} = 0
   . \]
Similarly
\[ \Wd_{11' 12} + \Wd_{1222'} = \RR_{11' 12} - \frac{r_{1' 2}}{2} + \RR_{1222'} +
   \frac{r_{12'}}{2} = \RR_{11' 12} + \RR_{1222'} - r_{1' 2} = 0 \]
\[ \Wd_{11' 12'} + \Wd_{12' 22'} = \RR_{11' 12'} - \frac{r_{1' 2'}}{2} + \RR_{12' 22'}
   + \frac{r_{12}}{2} = \RR_{11' 12'} + \RR_{12' 22'} + r_{12} = 0 \]
which proves that $\Wd$ is block-diagonal, i.e.\ commutes with the Hodge
operator.

\medskip

Let us now look more closely at the $\Wd^-$ term
\[ \left(\begin{array}{ccc}
     \Wd_{11' 11'} + \Wd_{22' 22'} + 2 \Wd_{11' 22'} & 0 & 0\\
     & 2 ( \Wd_{1212} + \Wd_{121' 2'}) & 2 ( \Wd_{1212'} + \Wd_{121' 2})\\
     &  & 2 ( \Wd_{12' 12'} + \Wd_{12' 1' 2})
   \end{array}\right) 
\]
\begin{multline*}
  \Wd_{11' 11'} + \Wd_{22' 22'} + 2 \Wd_{11' 22'} 
  \\ 
  =  \RR_{11' 11'} + r_{11} -
  \frac{\text{Scal}}{6} + \RR_{22' 22'} + r_{22} - \frac{\text{Scal}}{6} + 2
  \RR_{11' 22'}
  \\
   =  \RR_{11' 11'} + \RR_{22' 22'} + 2 \RR_{11' 22'} + \frac{\text{Scal}}{2} -
  \frac{\text{Scal}}{3}
  \\
   =  \RR_{11' 11'} + \RR_{22' 22'} + 2 ( - \RR_{1212} + \RR_{12' 12'}) +
  \frac{\text{Scal}}{6} = - \frac{\text{Scal}}{3}
\end{multline*}
where we have used the first Bianchi identity (and the invariance of
$\Rm$)
\[ \RR_{11' 22'} = - \RR_{1' 212'} - \RR_{211' 2'} = \RR_{12' 12'} - \RR_{1212} . \]
\begin{eqnarray*}
  \Wd_{1212} + \Wd_{121' 2'} 
  & = & \RR_{1212} - \frac{r_{22} + r_{11}}{2} +
  \frac{\text{Scal}}{6} + \RR_{121' 2'} 
  \\
  & = & \RR_{1212} - \frac{\text{Scal}}{4} +
  \frac{\text{Scal}}{6} + \RR_{121' 2'} = - \frac{\text{Scal}}{12}
\end{eqnarray*}
\begin{eqnarray*}
  \Wd_{12' 12'} + \Wd_{12' 1' 2} & = & \RR_{12' 12'} + \frac{\text{Scal}}{4} -
  \frac{\text{Scal}}{6} + \RR_{12' 1' 2} = \frac{\text{Scal}}{12}
\end{eqnarray*}
\[ \Wd_{1212'} + \Wd_{121' 2} = \RR_{1212'} - \frac{r_{22'}}{2} + \RR_{121' 2} -
   \frac{r_{11'}}{2} = 0 . \]
Finally,

\[ 
\Wd^- = \text{Scal} \left(\begin{array}{ccc}
     - 1 / 3 &  & \\
     & - 1 / 6 & \\
     &  & 1 / 6
   \end{array}\right) \]
vanishes if and only if $\text{Scal} = 0$. (The Reader will check that this matrix is indeed traceless
w.r.t.\ the pseudo-metric $g$.)


\end{document}